\documentclass[reqno]{amsart}
\usepackage{amssymb,latexsym}
\usepackage{amsmath}
\usepackage{amsthm}
\usepackage{graphicx}
\usepackage{hyperref}
\usepackage{titletoc}
\numberwithin{equation}{section}
\newtheorem{theorem}{Theorem}[section]

\newtheorem{lemma}[theorem]{Lemma}

\theoremstyle{definition}

\begin{document}
\title{Classification of the stable solution to biharmonic problems in large dimensions}

\author{  Juncheng Wei }
\address{ Department of Mathematics,
Chinese University of Hong Kong, Shatin, Hong Kong  }
\email{wei@math.cuhk.edu.hk}
\author{ Xingwang Xu}
\address{Department of Mathematics, National University of Singapore
, Singapore 119076, Republic of Singapore}
\email{matxuxw@nus.edu.sg}
\author{  Wen Yang }
\address{ Department of Mathematics,
Chinese University of Hong Kong, Shatin, Hong Kong  }
\email{wyang@math.cuhk.edu.hk}

\begin{abstract}
We give a new bound on the exponent for the nonexistence of stable solutions
to the biharmonic problem
$$\Delta^2u=u^p,\quad u>0~\operatorname{in}~\mathbb{R}^n $$
where $p>1, n \geq 20$.
\end{abstract}

\date{}
\maketitle

\section{Introduction}
Of concern  is the following biharmonic equation
\begin{equation}
\label{1.1}
\Delta^2u=u^p,\quad u>0~\operatorname{in}~\mathbb{R}^n
\end{equation}
where $n\geq 5$ and $p>1$. Let
\begin{equation}
\label{1.2}
\Lambda_u(\varphi):=\int_{\mathbb{R}^n}|\Delta\varphi|^2dx-p\int_{\mathbb{R}^n}u^{p-1}\varphi^2dx,\quad \forall~ \varphi\in H^2(\mathbb{R}^n).
\end{equation}
The Morse index of a classical solution to (\ref{1.1}), $ind(u)$ is defined as the maximal dimension of all subspaces of $E_{\mathbb{R}^n}:=H^2(\mathbb{R}^n)$ such that $\Lambda_u(\varphi)<0$ in $E_{\mathbb{R}^n}\setminus \{0\}$. We say $u$ is a stable solution to (\ref{1.1}) if $\Lambda_u(\varphi)\geq 0$ for any test function $\varphi \in H^2(\mathbb{R}^n),$ i.e., the Morse index is zero.

In the first part, we obtain the following classification result of stable solution to (\ref{1.1}).
\begin{theorem}
\label{theorem1.1}
Let $n\geq 5.$
\begin{enumerate}
  \item $n\leq 8$ and any $1<p<\infty$, the equation (\ref{1.1}) has no stable solution.
  \item $9\leq n\leq 19$, there exist $\varepsilon_n>0$ such that for any $1<p<\frac{n}{n-8}+\varepsilon_n,$ the equation (\ref{1.1}) has no stable solution.
  \item $20\leq n$ and $1<p<1+\frac{8p^*}{n-4},$ the equation (\ref{1.1}) has no stable solution.
\end{enumerate}
\end{theorem}
Here $p^*$ stands for the smallest real root which is greater than $\frac{n-4}{n-8}$ of the following algebraic equation
\begin{align*}
& 512(2-n)x^6+4(n^3-60n^2 +670n-1344)x^5-2(13n^3-424n^2+3064n-5408)x^4\nonumber\\
& +2(27n^3-572n^2 + 3264n-5440)x^3- (49n^3-772n^2 +3776n-5888)x^2\nonumber\\
& +4(5n^3-66n^2+288n-416)x-3(n^3 -12n^2 + 48 n - 64) = 0.
\end{align*}

Let us recall that for the second order problem
\begin{equation}
\label{1.3} \Delta u+ u^p =0\quad
u>0~\operatorname{in}~\mathbb{R}^n,~p>1,
\end{equation}
 Farina \cite{f} gave a complete classification of all  finite Morse index solutions.  The main result of \cite{f} is that no stable solution exists to (\ref{1.3}) if either $n \leq 10, p>1$ or $n\geq11,~p< p_{JL}$. Here $p_{JL}$ represents the well-known Joseph-Ludgren exponent, see \cite{gnw}. On the other hand, stable radial solution exists for $p \geq p_{JL}$.

In the fourth order case, the nonexistence of positive solutions to (\ref{1.1}) are showed if $p<\frac{n+4}{n-4},$ and all entire solutions are classified if $p=\frac{n+4}{n-4}$, see \cite{l}, \cite{wx}.  For $ p >\frac{n+4}{n-4}$,  the radially symmetric solutions to (\ref{1.1}) are completely classified in \cite{fgp}, \cite{gg} and \cite{gw2}. The radial solutions are shown to be stable if and only if  $p\geq p'_{JL}$ and $n\geq 13$ where $p'_{JL}$ stands for the corresponding Joseph-Lundgren exponent (see\cite{fgp}, \cite{gg}).  In the general case, in \cite{wy}, they showed the nonexistence of stable or finite Morse index solutions when either $n \leq 8, p>1$ or $ n \geq 9, p\leq \frac{n}{n-8}$. In dimensions $ n \geq 9$, a perturbation argument is used to show the nonexistence of stable solutions for $ p<\frac{n}{n-8}+\epsilon_n$ for some $\epsilon_n >0$. No explicit value of $\epsilon_n$ is given. The proof of \cite{wy} follows earlier idea of Cowan-Esposito-Ghoussoub \cite{cen} in which similar problem in a bounded domain was studied.

In the second order case, the proof of Farina uses basically the Moser iterations: namely multiply the equation (\ref{1.3}) by the power of $u$, like $ u^q, q>1$. Moser iteration works because of the following simple identity
\begin{equation*}
\int_{\mathbb{R}^n} u^q (-\Delta u)= \frac{4q}{(q+1)^2} \int_{\mathbb{R}^n} |\nabla u^{\frac{q+1}{2}}|^2, \forall u \in C_0^1 (\mathbb{R}^n).
\end{equation*}
In the fourth order case, such equality does not hold, and in fact we have
\begin{equation*}
\int_{\mathbb{R}^n} u^q (\Delta^2u)= \frac{4q}{(q+1)^2}  \int_{\mathbb{R}^n} |\Delta u^{\frac{q+1}{2}}|^2 -q (q-1)^2 \int_{\mathbb{R}^n} u^{q-3} |\nabla u|^4, \forall u \in C_0^2 (\mathbb{R}^n).
\end{equation*}
The additional term $\int_{\mathbb{R}^n} u^{q-3} |\nabla u|^4$
makes the Moser iteration argument difficult to use. In \cite{wy}, they used instead  the new test function $ -\Delta u$ and showed that $ \int_{\mathbb{R}^2} |\Delta u|^2$ is bounded. Thus the exponent $ \frac{n}{n-8}$ is obtained. In this paper, we use the Moser iteration for the fourth order problem and give a control on the term $\int_{\mathbb{R}^n} u^{q-3} |\nabla u|^4$ (Lemma 2.3). As a result, we obtain a better exponent $ \frac{n}{n-8}+\epsilon_n$ where $\epsilon_n$ is explicitly given. As far as we know, this seems to be the first result for Moser iteration for fourth order problem.

In the second part, we show  that the same idea can be used to show the regularity of  the extremal solutions to
\begin{equation}
\label{1.4}
\left\{
  \begin{array}{ll}
    \Delta^2u=\lambda(u+1)^p,~\lambda>0 & \operatorname{in}\quad \Omega\\
    u>0 , & \operatorname{in}\quad \Omega   \\
    u=\Delta u=0, & \operatorname{on}\quad \partial\Omega
  \end{array}
\right.
\end{equation}
where $\Omega $ is a smooth and bounded  convex domain in $\mathbb{R}^n$.

For problem (\ref{1.4}), it is known that
there exists a critical value $\lambda^{*}>0$ depending on $p>1$ and $\Omega$ such that
\begin{itemize}
  \item If $\lambda\in (0,\lambda^{*})$, equation (\ref{1.4}) has a minimal and classical solution which is stable;
  \item If $\lambda=\lambda^{*}$, a unique weak solution, called the extremal solution $u^*$ exists for equation (\ref{1.4});
  \item No weak solution of equation (\ref{1.4}) exists whenever $\lambda>\lambda^{*}.$
\end{itemize}
Our second  result is the following.
\begin{theorem}
\label{theorem1.2}
The extremal solutions $u^*$, the unique solution of (\ref{1.4}) when $\lambda=\lambda^*$ is bounded provided that
\begin{enumerate}
  \item $n\leq 8$ and any $1<p<\infty$,
  \item $9\leq n\leq 19$, there exist $\varepsilon_n>0$ such that for any $1<p<\frac{n}{n-8}+\varepsilon_n,$
  \item $20\leq n$ and $1<p<1+\frac{8p^*}{n-4}.\quad (\operatorname{where}~p^*~\operatorname{ is~defined~as~above})$
\end{enumerate}
\end{theorem}

This paper is organized as follows. We prove Theorem \ref{theorem1.1} and Theorem \ref{theorem1.2} respectively in section 2 and section 3.  Some technical inequalities are given in the appendix.

 \bigskip

\setcounter{equation}{0}
\section{Proof of Theorem 1.1}
In this section, we  prove  Theorem \ref{theorem1.1} through a series of Lemmas. First of all, we have following.
\begin{lemma}
\label{lemma2.1}
For any $\varphi \in C_0^4(\mathbb{R}^n),\varphi \geq 0$, $\gamma>1$ and $\varepsilon>0$ arbitrary small number,  we have
\begin{equation}
\label{2.1}
\int_{\mathbb{R}^n}(\Delta(u^{\gamma}\varphi^{\gamma}))^2\leq\int_{\mathbb{R}^n}((\Delta
u^{\gamma}\varphi^{\gamma})^2 + \varepsilon|\nabla
u|^4\varphi^{2\gamma}u^{2\gamma-4}+Cu^{2\gamma}\|\nabla^4(\varphi^{2\gamma})\|),
\end{equation}
\begin{equation}
\label{2.7}
\int_{\mathbb{R}^n}(\Delta(u^{\gamma}\varphi^{\gamma}))^2\geq\int_{\mathbb{R}^n}((\Delta
u^{\gamma}\varphi^{\gamma})^2 - \varepsilon|\nabla
u|^4\varphi^{2\gamma}u^{2\gamma-4}-Cu^{2\gamma}\|\nabla^4(\varphi^{2\gamma})\|),
\end{equation}
\begin{align}
\label{2.8}
\int_{\mathbb{R}^n}((u^{\gamma})_{ij})^2\varphi^{2\gamma}\leq&
\int_{\mathbb{R}^n}((u^{\gamma}\varphi^{\gamma})_{ij})^2
+\varepsilon\int_{\mathbb{R}^n}|\nabla u|^4u^{2\gamma-4}\varphi^{2\gamma}\nonumber\\
&+C\int_{\mathbb{R}^n}u^{2\gamma}\|\nabla^4(\varphi^{2\gamma})\|,
\end{align}
where $C$ is a positive number only depends on $\gamma,\varepsilon$
and $\|\nabla^4(\varphi^{2\gamma})\|$ is defined by
\begin{equation*}
\|\nabla^4(\varphi^{2\gamma})\|^2 = \varphi^{-2\gamma} |\nabla
\varphi^\gamma|^4 + |\varphi^\gamma (\Delta^2 \varphi^\gamma)| +
 |\nabla^2 \varphi^\gamma|^2.
\end{equation*}
\end{lemma}
In the following, unless otherwise,  the constant $C$  in this section always denotes  a positive number which  may change term by term  but only depends on $\gamma, \varepsilon$; the constant $C_0>0$ changes term by term, but independent of $\varepsilon.$
\begin{proof}
Since $\varphi$ has compact support, we can freely use the
integration by parts without mentioning the boundary term. First, by
direct calculations, we get
\begin{align}
\label{2.5}(\Delta(u^{\gamma}\varphi^{\gamma}))^2=&[(\Delta
u^{\gamma})\varphi^{\gamma}]^2+4\nabla
u^{\gamma}\nabla\varphi^{\gamma}\Delta
\varphi^{\gamma}u^{\gamma}+4\nabla
u^{\gamma}\nabla\varphi^{\gamma}\Delta u^{\gamma}\varphi^{\gamma}\nonumber\\
&+4(\nabla u^{\gamma}\nabla\varphi^{\gamma})^2 +2\Delta
u^{\gamma}u^{\gamma}\Delta\varphi^{\gamma}\varphi^{\gamma}+u^{2\gamma}(\Delta\varphi^{\gamma})^2.
\end{align}
We now may need to deal with the third and fifth term on the right
hand side of the above equality up to the integration both sides.

For the third term,
\begin{align*}
& \int_{\mathbb{R}^n}\Delta u^{\gamma}\nabla
u^{\gamma}\nabla\varphi^{\gamma}\varphi^{\gamma} = -
\int_{\mathbb{R}^n}(u^{\gamma})_i(u^{\gamma})_{ij}(\varphi^{\gamma})_j\varphi^{\gamma}\\
&-\int_{\mathbb{R}^n}(u^{\gamma})_i(u^{\gamma})_{j}(\varphi^{\gamma})_{ij}\varphi^{\gamma}-
\int_{\mathbb{R}^n}(u^{\gamma})_i(u^{\gamma})_{j}(\varphi^{\gamma})_j(\varphi^{\gamma})_i,
\end{align*}
where $f_i=\frac{\partial f}{\partial x_i}$ and
$f_{ij}=\frac{\partial^2f}{\partial x_j\partial x_i}.$  (Here and in
the sequel, we use the Einstein summation convention: an index
occurring twice in a product is to be summed from 1 up to the space
dimension, e.g., $u_iv_i=\sum_{i=1}^{n}u_iv_i,
\partial_i(u_iu_j\varphi_j)=\sum_{1\leq i,j\leq
n}\partial_i(u_iu_j\varphi_j)$.)  The first term on the right hand
side of the previous equation can be estimated as
\begin{align*}
2\int_{\mathbb{R}^n}(u^{\gamma})_i(u^{\gamma})_{ij}(\varphi^{\gamma})_j\varphi^{\gamma}=&
\int_{\mathbb{R}^n}\partial_j((u^{\gamma})_i(u^{\gamma})_{i}(\varphi^{\gamma})_j\varphi^{\gamma})-
\int_{\mathbb{R}^n}((u^{\gamma})_i)^2(\varphi^{\gamma})_{jj}\varphi^{\gamma}\\
&-\int_{\mathbb{R}^n}((u^{\gamma})_i)^2(\varphi^{\gamma})_{j}(\varphi^{\gamma})_j.
\end{align*}
Combining these two equalities, we get
\begin{align*}
2\int_{\mathbb{R}^n}\Delta u^{\gamma}\nabla u^{\gamma}\nabla\varphi^{\gamma}\varphi^{\gamma}=&
-\int_{\mathbb{R}^n}\partial_j((u^{\gamma})_i(u^{\gamma})_{i}(\varphi^{\gamma})_j\varphi^{\gamma})\\
&-\int_{\mathbb{R}^n}2(u^{\gamma})_i(u^{\gamma})_{j}(\varphi^{\gamma})_{ij}\varphi^{\gamma}
-\int_{\mathbb{R}^n}2(u^{\gamma})_i(u^{\gamma})_{j}(\varphi^{\gamma})_j(\varphi^{\gamma})_i\\
&+\int_{\mathbb{R}^n}((u^{\gamma})_i)^2(\varphi^{\gamma})_{jj}\varphi^{\gamma}
+\int_{\mathbb{R}^n}((u^{\gamma})_i)^2(\varphi^{\gamma})_{j}(\varphi^{\gamma})_j.
\end{align*}
Up to the short form of notation, thus we obtain
\begin{align}
\label{2.6} 4\int_{\mathbb{R}^n}\Delta u^{\gamma}\nabla
u^{\gamma}\nabla\varphi^{\gamma}\varphi^{\gamma}=
&2\int_{\mathbb{R}^n}|\nabla u^{\gamma}|^2\Delta
\varphi^{\gamma}\varphi^{\gamma}
+2 \int_{\mathbb{R}^n}|\nabla u^{\gamma}|^2|\nabla \varphi^{\gamma}|^2\nonumber\\
&-4\int_{\mathbb{R}^n}(u^{\gamma})_i(u^{\gamma})_j(\varphi^{\gamma})_{ij}\varphi^{\gamma}
-4 \int_{\mathbb{R}^n}(<\nabla u^{\gamma}, \nabla \varphi^{\gamma}>
)^2 .
\end{align}

For the fifth term on the right hand side of Equation (\ref{2.5}) we
have
\begin{align}
\label{2.6a} \int_{\mathbb{R}^n}\Delta
u^{\gamma}u^{\gamma}\Delta\varphi^{\gamma}\varphi^{\gamma}=&
-\int_{\mathbb{R}^n}u^{\gamma}<\nabla u^{\gamma},\nabla(\Delta\varphi^{\gamma})>\varphi^{\gamma}\nonumber\\
&-\int_{\mathbb{R}^n}<\nabla u^{\gamma}, \nabla\varphi^{\gamma}>
u^{\gamma}\Delta\varphi^{\gamma}-\int_{\mathbb{R}^n}|\nabla
u^{\gamma}|^2\Delta\varphi^{\gamma}\varphi^{\gamma}.
\end{align}

Combining Equations (\ref{2.5}), (\ref{2.6}) and (\ref{2.6a}), one
obtains
\begin{align}
\label{2.6b} & \int_{\mathbb{R}^n} (\Delta(u^\gamma
\varphi^\gamma))^2  - \int_{\mathbb{R}^n} (\Delta u^\gamma)^2
\varphi^{2\gamma} \nonumber\\
& =  2\int_{\mathbb{R}^n} |\nabla u^\gamma|^2
|\nabla\varphi^\gamma|^2 - 4 \int_{\mathbb{R}^n} \varphi^\gamma
(\nabla^2\varphi^\gamma(\nabla u^\gamma, \nabla u^\gamma))\nonumber\\
&  + \int_{\mathbb{R}^n} u^{2\gamma} \varphi^\gamma
\Delta^2(\varphi^\gamma) - 2\int_{\mathbb{R}^n} u^{2\gamma} (\Delta
\varphi^\gamma)^2.
\end{align}

Now by the Young equality, for any $\epsilon > 0$, there exists a
constant $C(\epsilon)$ such that
$$ |\nabla u^\gamma|^2 |\nabla \varphi^\gamma|^2 \le
\frac{\epsilon}{4} |\nabla u^\gamma|^4 u^{-2\gamma}
\varphi^{2\gamma} + C(\epsilon) |\nabla\varphi^\gamma|^4 u^{2\gamma}
\varphi^{-2\gamma}$$ and
$$|\varphi^\gamma
(\nabla^2\varphi^\gamma(\nabla u^\gamma, \nabla u^\gamma))| \le
\frac{\epsilon}{8} |\nabla u^\gamma|^4 u^{-2\gamma}
\varphi^{2\gamma} + C(\epsilon) u^{2\gamma}
|\nabla^2\varphi^\gamma|^2.$$

Thus by the equation (\ref{2.6b}), together with these two
estimates, one gets:
$$ |\int_{\mathbb{R}^n} (\Delta(u^\gamma
\varphi^\gamma))^2  - \int_{\mathbb{R}^n} (\Delta u^\gamma)^2
\varphi^{2\gamma}| \le \epsilon \int_{\mathbb{R}^n}|\nabla
u^\gamma|^4 u^{-2\gamma} \varphi^{2\gamma} + 6 C(\epsilon)
\int_{\mathbb{R}^n} u^{2\gamma} \|\nabla^4 \varphi^\gamma\|^2.$$

Thus the estimates (\ref{2.1}) and (\ref{2.7}) follow from this
easily by observing that we always have $|\Delta \varphi^\gamma|^2
\le |\nabla^2\varphi^\gamma|^2$.

Now observe that $|\nabla^2u^\gamma|^2 \varphi^{2\gamma} =
[\frac{1}{2} \Delta|\nabla u^\gamma|^2 - <\nabla u^\gamma, \nabla
\Delta u^\gamma>]\varphi^{2\gamma}$. Thus up to the integration by
parts, with the help of the equation (\ref{2.6}) and the estimates
we just proved, the estimate (\ref{2.8}) also follows except one
should notice that $\int_{\mathbb{R}^n} (\Delta
(u^\gamma\varphi^\gamma))^2 = \int_{\mathbb{R}^n} |\nabla^2
(u^\gamma\varphi^\gamma)|^2.$ Thus Lemma \ref{lemma2.1} follows.
\end{proof}

Let us return to the equation
\begin{equation}
\label{2.9}
\Delta^2u=u^p,\quad u>0~\operatorname{in}~\mathbb{R}^n.
\end{equation}
Fix $q=2\gamma -1>0$ and $\gamma >1$. Let $\varphi \in C_0^\infty (\mathbb{R}^n)$. Multiplying  (\ref{2.9}) by $u^{q}\varphi^{2\gamma}$  and integration by parts,  we obtain
\begin{equation}
\label{2.10}
\int_{\mathbb{R}^n}\Delta u\Delta(u^q\varphi^{2\gamma})=\int_{\mathbb{R}^n}u^{p+q}\varphi^{2\gamma}.
\end{equation}
For the left hand side of (\ref{2.10}), we have the following lemma.
\begin{lemma}
\label{lemma2.2} For any $\varphi\in C_0^{\infty}(\mathbb{R}^n)$
with $\varphi\geq 0$,  for any $\varepsilon>0$ and $\gamma$ with $q$
defined above, there exists a positive constant $C$ such that
\begin{align}
\label{2.11}
\int_{\mathbb{R}^n}\frac{\gamma^2}{q}\Delta u\Delta(u^{q}\varphi^{2\gamma})\geq
&\int_{\mathbb{R}^n}(\Delta u^{\gamma}\varphi^{\gamma})^2-\int_{\mathbb{R}^n} C u^{2\gamma}\|\nabla^4(\varphi^{2\gamma})\|\nonumber\\
&-\int_{\mathbb{R}^n}(\gamma^2(\gamma-1)^2 +
\varepsilon)u^{2\gamma-4}|\nabla u|^4\varphi^{2\gamma}.
\end{align}
\end{lemma}
\begin{proof}
First, by direct computations, we obtain
\begin{align*}
\Delta u\Delta(u^{2\gamma-1}\varphi^{2\gamma})=&\Delta u((2\gamma-1)u^{2\gamma-2}\Delta u\varphi^{2\gamma}+2(2\gamma-1) u^{2\gamma-2}\nabla u\nabla(\varphi^{2\gamma})\\
&+(2\gamma-1)(2\gamma-2)u^{2\gamma-3}|\nabla u|^2\varphi^{2\gamma}+u^{2\gamma-1}\Delta\varphi^{2\gamma}),
\end{align*}
\begin{align*}
(\Delta u^{\gamma}\varphi^{\gamma})^2=&\gamma^2u^{2\gamma-2}(\Delta u)^2\varphi^{2\gamma}+\gamma^2(\gamma-1)^2u^{2\gamma-4}|\nabla u|^4\varphi^{2\gamma}\\
&+2(\gamma-1)\gamma^2u^{2\gamma-3}|\nabla u|^2\Delta u\varphi^{2\gamma}.
\end{align*}
Combining the above two identities, we get
\begin{align}
\label{2.12}
\frac{\gamma^2}{q}\Delta u\Delta(u^q\varphi^{2\gamma})=&(\Delta u^{\gamma}\varphi^{\gamma})^2+2\gamma^2u^{2\gamma-2}\Delta u\nabla u\nabla \varphi^{2\gamma}+\frac{\gamma^2}{q}u^{2\gamma-1}\Delta u\Delta \varphi^{2\gamma}\nonumber\\&-\gamma^2(\gamma-1)^2u^{2\gamma-4}|\nabla u|^4\varphi^{2\gamma}.
\end{align}
For the term  $u^{2\gamma-2}\Delta u\nabla u\nabla \varphi^{2\gamma}$, we have
\begin{align*}
u^{2\gamma-2}\Delta u\nabla u\nabla\varphi^{2\gamma}=&\partial_i(u^{2\gamma-2}u_iu_j(\varphi^{2\gamma})_j)-
(2\gamma-2)u^{2\gamma-3}(u_i)^2u_j(\varphi^{2\gamma})_j\\
&-u^{2\gamma-2}u_iu_{ij}(\varphi^{2\gamma})_j-u^{2\gamma-2}u_iu_j(\varphi^{2\gamma})_{ij}.
\end{align*}
We can regroup the term $u^{2\gamma-2}u_iu_{ij}(\varphi^{2\gamma})_j$ as
\begin{align*}
2u^{2\gamma-2}u_iu_{ij}(\varphi^{2\gamma})_j=&\partial_j(u^{2\gamma-2}(u_i)^2(\varphi^{2\gamma})_j)-
(2\gamma-2)u^{2\gamma-3}u_j(u_i)^2(\varphi^{2\gamma})_j\\
&-u^{2\gamma-2}(u_i)^2(\varphi^{2\gamma})_{jj}.
\end{align*}
Therefore we get
\begin{align}
\label{2.13}
2u^{2\gamma-2}\Delta u\nabla u\nabla\varphi^{2\gamma}=&2
\partial_i(u^{2\gamma-2}u_iu_j(\varphi^{2\gamma})_j)-\partial_j(u^{2\gamma-2}(u_i)^2(\varphi^{2\gamma})_j)\nonumber\\
&-(2\gamma-2)u^{2\gamma-3}(u_i)^2u_j(\varphi^{2\gamma})_j+u^{2\gamma-2}(u_i)^2(\varphi^{2\gamma})_{jj}\nonumber\\
&-2u^{2\gamma-2}u_iu_j(\varphi^{2\gamma})_{ij}
\end{align}
For the last three terms on the right hand side of (\ref{2.13}),
applying Young's inequality, we get
\begin{align*}
|u^{2\gamma-3}(u_i)^2u_j(\varphi^{2\gamma})_j|\leq
\frac{\varepsilon}{6 \gamma^2 (\gamma - 1)} u^{2\gamma-4}|\nabla
u|^4\varphi^{2\gamma}+ C u^{2\gamma}\|\nabla^4(\varphi^{2\gamma})\|,
\\
|u^{2\gamma-2}(u_i)^2(\varphi^{2\gamma})_{jj}|\leq
\frac{\varepsilon}{6\gamma^2} u^{2\gamma-4}|\nabla
u|^4\varphi^{2\gamma}+ C u^{2\gamma}\|\nabla^4(\varphi^{2\gamma})\|,
\\
|u^{2\gamma-2}u_iu_j(\varphi^{2\gamma})_{ij}|\leq
\frac{\varepsilon}{6 \gamma^2} u^{2\gamma-4}|\nabla
u|^4\varphi^{2\gamma}+Cu^{2\gamma}\|\nabla^4(\varphi^{2\gamma})\|.
\end{align*}
By the above three inequalities and  (\ref{2.13}), we have
\begin{equation}
\label{2.14} \int_{\mathbb{R}^n}2\gamma^2u^{2\gamma-2}\Delta u\nabla
u\nabla \varphi^{2\gamma}\geq- \frac{\varepsilon}{2}
\int_{\mathbb{R}^n}u^{2\gamma-4}|\nabla u|^4\varphi^{2\gamma}-
C\int_{\mathbb{R}^n}u^{2\gamma}\|\nabla^4(\varphi^{2\gamma})\|.
\end{equation}
Similarly  we get
\begin{equation}
\label{2.15}
\int_{\mathbb{R}^n}\frac{\gamma^2}{q}u^{2\gamma-1}\Delta u\Delta
\varphi^{2\gamma}\geq -\frac{\varepsilon}{2}
\int_{\mathbb{R}^n}u^{2\gamma-4}|\nabla u|^4\varphi^{2\gamma}-
C\int_{\mathbb{R}^n}u^{2\gamma}\|\nabla^4(\varphi^{2\gamma})\|.
\end{equation}
The inequality (\ref{2.11}) follows from (\ref{2.12}), (\ref{2.14}) and (\ref{2.15}).
\end{proof}

 \medskip

As a result of (\ref{2.1}) and (\ref{2.11}), we have
\begin{align}
\label{2.16}
\int_{\mathbb{R}^n}\frac{\gamma^2}{q}\Delta u\Delta(u^{q}\varphi^{2\gamma})\geq&\int_{\mathbb{R}^n}(\Delta( u^{\gamma}\varphi^{\gamma}))^2-\int_{\mathbb{R}^n}Cu^{2\gamma}\|\nabla^4(\varphi^{2\gamma})\|\nonumber\\
&-\int_{\mathbb{R}^n}(\gamma^2(\gamma-1)^2 +
\varepsilon)u^{2\gamma-4}|\nabla u|^4\varphi^{2\gamma}.
\end{align}
Next we estimate the  most difficult term $\int_{\mathbb{R}^n}u^{2\gamma-4}|\nabla u|^4\varphi^{2\gamma}$ appeared in (\ref{2.16}). This is the key step in proving Theorem \ref{theorem1.1}.
\begin{lemma}
\label{lemma2.3} If u is the classical solution to the bi-harmonic
equation (\ref{2.9}), and $\varphi$ is defined as above, then for
any sufficiently small $\varepsilon > 0$, we have the following
inequality
\begin{align}
\label{2.17}
(\frac{1}{2} - \varepsilon)\int_{\mathbb{R}^n}u^{2\gamma-4}|\nabla u|^4\varphi^{2\gamma}\leq& \frac{2}{\gamma^2}
\int_{\mathbb{R}^n}(\Delta (u^{\gamma}\varphi^{\gamma}))^2+\int_{\mathbb{R}^n}Cu^{2\gamma}\|\nabla^4(\varphi^{2\gamma})\|\nonumber\\
&-\int_{\mathbb{R}^n}\frac{4}{(4\gamma-3+p)(p+1)}u^{2\gamma+p-1}\varphi^{2\gamma}.
\end{align}
\end{lemma}
\begin{proof}
It is easy to see that
\begin{equation}
\label{2.18} \int_{\mathbb{R}^n}u^{2\gamma-4}|\nabla
u|^4\varphi^{2\gamma}= \frac{1}{\gamma^4}\int_{\mathbb{R}^n}
u^{-2\gamma}|\nabla u^{\gamma}|^4\varphi^{2\gamma},
\end{equation}
and
\begin{align}
\label{2.19}
\int_{\mathbb{R}^n}u^{-2\gamma}|\nabla u^{\gamma}|^4\varphi^{2\gamma}=&
\int_{\mathbb{R}^n}u^{-2\gamma}|\nabla u^{\gamma}|^2\nabla u^{\gamma}\nabla u^{\gamma}\varphi^{2\gamma}\nonumber\\
=&
\int_{\mathbb{R}^n}-\nabla u^{-\gamma}|\nabla u^{\gamma}|^2\nabla u^{\gamma}\varphi^{2\gamma}\nonumber\\
=&\int_{\mathbb{R}^n}\frac{|\nabla u^{\gamma}|^2\Delta u^{\gamma}\varphi^{2\gamma}}{u^{\gamma}}+
\int_{\mathbb{R}^n}\frac{\nabla(|\nabla u^{\gamma}|^2)\nabla u^{\gamma}\varphi^{2\gamma}}{u^{\gamma}}\nonumber\\
&+ \int_{\mathbb{R}^n}\frac{|\nabla u^{\gamma}|^2\nabla
u^{\gamma}\nabla\varphi^{2\gamma}}{u^{\gamma}},
\end{align}
where the last step is integration by parts. For the first term  in the last part of the above equality, we have
\begin{align}
\label{2.20}
\int_{\mathbb{R}^n}\frac{|\nabla u^{\gamma}|^2\Delta u^{\gamma}\varphi^{2\gamma}}{u^{\gamma}}=
\gamma^3\int_{\mathbb{R}^n}((\gamma-1)u^{2\gamma-4}|\nabla u|^4\varphi^{2\gamma}+
u^{2\gamma-3}|\nabla u|^2\Delta u\varphi^{2\gamma}).
\end{align}
Substituting (\ref{2.20}) into (\ref{2.19}), and combining with (\ref{2.18}), we obtain
\begin{align}
\label{2.21}
\int_{\mathbb{R}^n}u^{2\gamma-4}|\nabla u|^4\varphi^{2\gamma}=&\int_{\mathbb{R}^n}\frac{1}{\gamma^3}\frac{\nabla(|\nabla u^{\gamma}|^2)\nabla u^{\gamma}\varphi^{2\gamma}}{u^{\gamma}}+\int_{\mathbb{R}^n}u^{2\gamma-3}(|\nabla u|^2)\Delta u\varphi^{2\gamma}\nonumber\\
&+
\int_{\mathbb{R}^n}\frac{1}{\gamma^3}\frac{(|\nabla u^{\gamma}|^2)\nabla u^{\gamma}\nabla\varphi^{2\gamma}}{u^{\gamma}}.
\end{align}
The first term on the right hand side of (\ref{2.21}) can be
estimated as
\begin{align*}
u^{-\gamma} \nabla(|\nabla u^{\gamma}|^2)\nabla u^{\gamma}&=2 u^{-\gamma}((u^{\gamma})_{ij}(u^{\gamma})_i(u^{\gamma})_j)\\
&\leq 2\gamma(u^{\gamma})_{ij}(u^{\gamma})_{ij}+\frac{u^{-2\gamma}}{2\gamma}(u^{\gamma})_i(u^{\gamma})_j(u^{\gamma})_i(u^{\gamma})_j\\
&=2\gamma|\nabla^2u^{\gamma}|^2+\frac{u^{-2\gamma}}{2\gamma}|\nabla
u^{\gamma}|^4.
\end{align*}
As a consequence, we have
\begin{align}
\label{2.22}
\int_{\mathbb{R}^n}\frac{1}{\gamma^3}\frac{\nabla(|\nabla
u^{\gamma}|^2)\nabla u^{\gamma}\varphi^{2\gamma}}{u^{\gamma}} \leq &
\int_{\mathbb{R}^n}\frac{2}{\gamma^2}|\nabla^2
u^{\gamma}|^2\varphi^{2\gamma}
+\int_{\mathbb{R}^n}\frac{1}{2\gamma^4}\frac{|\nabla
u^{\gamma}|^4\varphi^{2\gamma}}{u^{2\gamma}}
\nonumber\\
\leq& \int_{\mathbb{R}^n}\frac{2}{\gamma^2}|\nabla^2
(u^{\gamma}\varphi^{\gamma})|^2+
\int_{\mathbb{R}^n}Cu^{2\gamma}\|\nabla^4(\varphi^{2\gamma})\|\nonumber\\
&+\int_{\mathbb{R}^n}\frac{1+4\gamma^2
\varepsilon}{2\gamma^4}\frac{|\nabla
u^{\gamma}|^4\varphi^{2\gamma}}{u^{2\gamma}}\nonumber\\
& = \int_{\mathbb{R}^n}\frac{2}{\gamma^2}(\Delta
(u^{\gamma}\varphi^{\gamma}))^2+
\int_{\mathbb{R}^n}Cu^{2\gamma}\|\nabla^4(\varphi^{2\gamma})\|\nonumber\\
&+\int_{\mathbb{R}^n}\frac{1+4\gamma^2
\varepsilon}{2\gamma^4}\frac{|\nabla
u^{\gamma}|^4\varphi^{2\gamma}}{u^{2\gamma}},
\end{align}
where we used (\ref{2.8}) in the last step.

For the second term on the right hand side of (\ref{2.21}), applying
the estimate (2.3) from \cite{wy}, i.e., $(\Delta
u)^2\geq\frac{2}{p+1}u^{p+1}$, and the fact that $\Delta u<0$ from
\cite{wx} or \cite{xu}, we have
\begin{align}
\label{2.24}
\int_{\mathbb{R}^n}u^{2\gamma-3}(|\nabla u|^2)\Delta u\varphi^{2\gamma}\leq&-\int_{\mathbb{R}^n}\sqrt{\frac{2}{p+1}}u^{2\gamma-3+\frac{p+1}{2}}(|\nabla u|^2)\varphi^{2\gamma}\nonumber\\
=&+\int_{\mathbb{R}^n}\frac{\sqrt{\frac{2}{p+1}}}{2\gamma-2+\frac{p+1}{2}}u^{2\gamma-2+\frac{p+1}{2}}\Delta u\varphi^{2\gamma}\nonumber\\
&+\int_{\mathbb{R}^n}\frac{\sqrt{\frac{2}{p+1}}}{2\gamma-2+\frac{p+1}{2}}u^{2\gamma-2+\frac{p+1}{2}}\nabla u\nabla\varphi^{2\gamma}.
\end{align}
Using the inequality $ -\Delta u \geq \sqrt{\frac{2}{p+1}} u^{\frac{p+1}{2}}$, we get
\begin{equation}
\label{2.25}
\int_{\mathbb{R}^n}\frac{\sqrt{\frac{2}{p+1}}}{2\gamma-2+\frac{2}{p+1}}u^{2\gamma-2+\frac{p+1}{2}}\Delta u\varphi^{2\gamma}\leq-\int_{\mathbb{R}^n}\frac{\frac{2}{p+1}}{2\gamma-2+\frac{p+1}{2}}u^{2\gamma+p-1}\varphi^{2\gamma}.
\end{equation}
On the other hand, for the second term on the right hand side of
(\ref{2.24}), we have
\begin{align}
\label{2.26}
\int_{\mathbb{R}^n}u^{2\gamma-2+\frac{p+1}{2}}\nabla u\nabla\varphi^{2\gamma}=&
-\int_{\mathbb{R}^n}\frac{1}{L}u^{2\gamma-1+\frac{p+1}{2}}\Delta\varphi^{2\gamma}\nonumber\\
=&-\int_{\{x|\Delta\varphi^{2\gamma}>0\}}\frac{1}{L}u^{2\gamma-1+\frac{p+1}{2}}\Delta\varphi^{2\gamma}\nonumber\\
&-\int_{\{x|\Delta\varphi^{2\gamma}\leq0\}}\frac{1}{L}u^{2\gamma-1+\frac{p+1}{2}}\Delta\varphi^{2\gamma},
\end{align}
where the first equality follows from integration by parts and $L=2\gamma-1+\frac{p+1}{2}$. As for the first term on the last part of (\ref{2.26}), using the inequality $ \Delta u \leq -\sqrt{\frac{2}{p+1}} u^{\frac{p+1}{2}}<0$, we have
\begin{align}
\label{100}
\frac{\sqrt{\frac{p+1}{2}}}{L}\int_{\{x|\Delta\varphi^{2\gamma}>0\}}u^{2\gamma-1}\Delta u
\Delta\varphi^{2\gamma}
\leq-\int_{\{x|\Delta\varphi^{2\gamma}>0\}}\frac{1}{L}u^{2\gamma-1+\frac{p+1}{2}}\Delta\varphi^{2\gamma}.
\end{align}
Similar to the proof of Lemma \ref{lemma2.1}, it is easy to get
\begin{equation}
\label{200}
|\int_{\{x|\Delta\varphi^{2\gamma}>0\}}\frac{\sqrt{\frac{p+1}{2}}}{L}u^{2\gamma-1}\Delta
u \Delta\varphi^{2\gamma}|\leq
\varepsilon\int_{\mathbb{R}^n}u^{2\gamma-4}|\nabla
u|^4\varphi^{2\gamma}+\int_{\mathbb{R}^n}Cu^{2\gamma}\|\nabla^4(\varphi^{2\gamma})\|.
\end{equation}

By (\ref{100}) and (\ref{200}), we have
\begin{equation}
\label{2.27}
|\int_{\{x|\Delta\varphi^{2\gamma}>0\}}\frac{1}{L}u^{2\gamma-1+\frac{p+1}{2}}\Delta\varphi^{2\gamma}|\leq
\varepsilon\int_{\mathbb{R}^n}u^{2\gamma-4}|\nabla
u|^4\varphi^{2\gamma}+\int_{\mathbb{R}^n}Cu^{2\gamma}\|\nabla^4(\varphi^{2\gamma})\|.
\end{equation}
Similarly, we also obtain
\begin{equation}
\label{2.28}
|\int_{\{x|\Delta\varphi^{2\gamma}\leq0\}}\frac{1}{L}u^{2\gamma-1+\frac{p+1}{2}}\Delta\varphi^{2\gamma}|\leq
\varepsilon\int_{\mathbb{R}^n}u^{2\gamma-4}|\nabla
u|^4\varphi^{2\gamma}+\int_{\mathbb{R}^n}Cu^{2\gamma}\|\nabla^4(\varphi^{2\gamma})\|.
\end{equation}
By (\ref{2.26}), (\ref{2.27}) and (\ref{2.28}), we have
\begin{align}
\label{2.29} |\int_{\mathbb{R}^n}u^{2\gamma-2+\frac{p+1}{2}}\nabla
u\nabla\varphi^{2\gamma}|\leq
\varepsilon\int_{\mathbb{R}^n}u^{2\gamma-4}|\nabla
u|^4\varphi^{2\gamma}+\int_{\mathbb{R}^n}Cu^{2\gamma}\|\nabla^4(\varphi^{2\gamma})\|.
\end{align}
Combining (\ref{2.24}), (\ref{2.25}) and (\ref{2.29}), we get the following inequality
\begin{align}
\label{2.30}
\int_{\mathbb{R}^n}u^{2\gamma-3}|\nabla u|^2\Delta u \varphi^{2\gamma}\leq&
\varepsilon\int_{\mathbb{R}^n}u^{2\gamma-4}|\nabla u|^4\varphi^{2\gamma}+\int_{\mathbb{R}^n}Cu^{2\gamma}\|\nabla^4(\varphi^{2\gamma})\|\nonumber\\
&-\int_{\mathbb{R}^n}\frac{4}{(4\gamma-3+p)(p+1)}u^{2\gamma+p-1}\varphi^{2\gamma}.
\end{align}
Finally, we apply Young's inequality to the third term on the right
hand side of (\ref{2.21}), and get
\begin{align}
\label{2.31}
\int_{\mathbb{R}^n}\frac{(|\nabla u^{\gamma}|^2)\nabla u^{\gamma}\nabla\varphi^{2\gamma}}{u^{\gamma}\gamma^3}
&=\int_{\mathbb{R}^n}u^{2\gamma-3}|\nabla u|^2\nabla u\nabla(\varphi^{2\gamma})\nonumber\\
&\leq\varepsilon\int_{\mathbb{R}^n}u^{2\gamma-4}|\nabla
u|^4\varphi^{2\gamma}+\int_{\mathbb{R}^n}Cu^{2\gamma}\|\nabla^4(\varphi^{2\gamma})\|.
\end{align}
By (\ref{2.21}), (\ref{2.22}), (\ref{2.30}) and (\ref{2.31}), we
finally obtain
\begin{align*}
(\frac{1}{2}- \varepsilon)\int_{\mathbb{R}^n}u^{2\gamma-4}|\nabla u|^4\varphi^{2\gamma}\leq &\frac{2}{\gamma^2}\int_{\mathbb{R}^n}(\Delta( u^{\gamma}\varphi^{\gamma}))^2+\int_{\mathbb{R}^n}Cu^{2\gamma}\|\nabla^4(\varphi^{2\gamma})\|\\
\\
&-\int_{\mathbb{R}^n}\frac{4}{(4\gamma-3+p)(p+1)}u^{2\gamma+p-1}\varphi^{2\gamma}.
\end{align*}
\end{proof}

By (\ref{2.10}), (\ref{2.16}) and (\ref{2.17}), since the number
$\varepsilon$ is arbitrary small in those three places, we have for
$\delta>0$ sufficiently small, the following inequality holds
\begin{align}
\label{2.32}
\int_{\mathbb{R}^n}(1-4(\gamma-1)^2-\delta)(\Delta( u^{\gamma}\varphi^{\gamma}))^2-&
\int_{\mathbb{R}^n}(\frac{\gamma^2}{2\gamma-1}-\frac{8\gamma^2(\gamma-1)^2}{(4\gamma-3+p)(p+1)})u^{p+2\gamma-1}\varphi^{2\gamma}
\nonumber\\
&\leq\int_{\mathbb{R}^n}C_{\delta}u^{2\gamma}\|\nabla^4(\varphi^{2\gamma})\|,
\end{align}
where $C_{\delta}$ is a positive constant depends on $\delta$ only.
Here, we need $1-4(\gamma-1)^2>0,$ since we have assumed that $\gamma>1$ in Lemma (\ref{lemma2.1}). So $\gamma$ is required be in $(1,\frac{3}{2})$. While we can choose $\delta$ small enough to make $1-4(\gamma-1)^2-\delta$ positive, by the stability property of function $u$, we obtain
\begin{equation}
\label{2.33}
\int_{\mathbb{R}^n}(E-p\delta)u^{p+q}\varphi^{2\gamma}
\leq\int_{\mathbb{R}^n}C_{\delta}u^{2\gamma}\|\nabla^4(\varphi^{2\gamma})\|,
\end{equation}
where $E$ is defined to be
\begin{equation}
E=p(1-4(\gamma-1)^2)-
\frac{\gamma^2}{q}+\frac{8\gamma^2(\gamma-1)^2}{(4\gamma-3+p)(p+1)}
\end{equation}
 Now we  take $\varphi=\eta^m$ with $m$ sufficiently large, and choose $\eta$ a cut-off function satisfying $0\leq\eta\leq1,~\eta=1$ for $|x|<R$ and $\eta=0$ for $|x|>2R$. By Young's inequality again, we have
\begin{align}
\label{2.34}
\int_{\mathbb{R}^n}u^{2\gamma}\|\nabla^4(\varphi^{2\gamma})\|
&\leq C_{\delta}R^{-4}\int_{\mathbb{R}^n}u^{2\gamma}\eta^{2\gamma m-4}\nonumber\\
&\leq C_{\delta, \epsilon}R^{-\frac{4}{1-\theta}}\int_{\mathbb{R}^n}u^2\eta^{2\gamma m-\frac{4}{1-\theta}}+
\epsilon C_{\delta}\int_{\mathbb{R}^n}u^{2\gamma+p-1}\eta^{2\gamma m},
\end{align}
where $C_{\delta, \epsilon}$ is a positive constant depends on $\delta,\epsilon$,  $\theta$ is a number such that $2(1-\theta)+(2\gamma+p-1)\theta=2\gamma$ so that $0<\theta<1$ for $2<2\gamma<2\gamma+p-1$. By (\ref{2.33}) and (\ref{2.34}), we get
\begin{equation}
\label{2.35}
(E-p\delta-\epsilon C_{\delta})\int_{\mathbb{R}^n}u^{p+2\gamma-1}\eta^{2\gamma m}
\leq
C_{\delta,\epsilon}R^{-\frac{4}{1-\theta}}\int_{\mathbb{R}^n}u^2\eta^{2\gamma m-\frac{4}{1-\theta}}.
\end{equation}
Since $\theta$ is strictly less than $1$ and will be fixed for given $\gamma,p$,  we can choose $m$ sufficiently large to make $2\gamma m-\frac{4}{1-\theta}>0$. On the other hand, if $E>0$, we can find small $\delta$ and then small $\epsilon$, such that $E-p\delta-\epsilon C_{\delta}>0$. Therefore, by the definition of function $\eta$ and (\ref{2.35}), we obtain
\begin{equation}
\label{2.36}
(E-p\delta-\epsilon C_{\delta})\int_{B_R}u^{p+2\gamma-1}
\leq
C_{\delta,\epsilon}R^{-\frac{4}{1-\theta}}\int_{B_{2R}}u^2.
\end{equation}
By \cite{wy}, we have $\int_{B_{2R}}u^2\leq CR^{n-\frac{8}{p-1}},$ as a result, the left hand side of (\ref{2.36}) is less equal than $C_{\delta,\epsilon}R^{n-\frac{8}{p-1}-\frac{4}{1-\theta}},$ which tends to $0$ as $R$ tends to $\infty$, provided the power $n-\frac{8}{p-1}-\frac{4}{1-\theta}$ is negative. By the definition of $\theta$, this is equivalent to $(p+2\gamma-1)>(p-1)\frac{n}{4}.$ So, if $(p+2\gamma-1)>(p-1)\frac{n}{4}$ and $E-p\delta-C_{\delta}\epsilon>0$, we have that $u\equiv0$.

Thus, we  have proved the nonexistence of the stable solution to (\ref{2.9}) if $p$ satisfies the condition $(p+2\gamma-1)>(p-1)\frac{n}{4}$ and $E>0$ (for $\delta, \epsilon$ are arbitrary small).
By Lemma \ref{lemma5.1},  the power $p$ can be  in the interval $(\frac{n}{n-8},1+\frac{8p^*}{n-4})$. Combining  with Theorem 1.1 of \cite{wy}, we prove the third part of Theorem \ref{theorem1.1}, i.e., for any $1<p<1+\frac{8p^*}{n-4},~n\geq 20$,  equation (\ref{2.9}) has no stable solution.
The first and second part of Theorem \ref{theorem1.1} is contained in Theorem 1.1 of \cite{wy}.
\vspace{1.5cm}

\setcounter{equation}{0}
\section{Proof of Theorem 1.2}
In this section, we give the proof of Theorem 1.2. We note that it is enough to consider stable solutions $u_{\lambda}$ to (\ref{1.4}) since $u^{*}=\mathrm{lim}_{\lambda\rightarrow\lambda^{*}}u_{\lambda}.$  Now we give a uniform bound for the stable solutions to (\ref{1.4}) when $0<d<\lambda<\lambda^*,$ where $d$ is a fixed positive constant from $(0,\lambda^*).$

First, we need to analyze the solution near the boundary.

\subsection{Regularity of the solution on the boundary}
In this subsection, we establish the regularity of the stable
solution with its derivative near the boundary of the following
equation:
\begin{equation}
\label{3.1}
\left\{
  \begin{array}{ll}
    \Delta^2u=\lambda(u+1)^p,~\lambda>0 & \operatorname{in}\quad \Omega\\
    u>0 , & \operatorname{in}\quad \Omega   \\
    u=\Delta u=0, & \operatorname{on}\quad \partial\Omega
  \end{array}
\right.
\end{equation}

\begin{theorem}
\label{theorem3.1} Let $\Omega$ be a bounded, smooth, and convex
domain. Then there exists a constant C (independent of $\lambda,u$)
and small positive number $\epsilon$, such that for stable solution
u to (\ref{3.1}) we have

\begin{equation}
\label{3.2}
u(x)<C,\quad \forall x\in \Omega_{\epsilon}:=\{z\in \Omega:~d(z,\partial\Omega)<\epsilon\}.
\end{equation}
\end{theorem}

\begin{proof}
This result is well-known. For the sake of completeness, we include a proof here. By Lemma 3.5 of \cite{cen}, we see that, there exists a constant $C$ independent of $\lambda, u$, such that
\begin{equation}
\label{3.3}
\int_{\Omega}(1+u)^pdx\leq C.
\end{equation}

We write Equation (\ref{3.1}) as

$$\left\{
  \begin{array}{ll}
    \Delta u+v=0, & \operatorname{in} \quad\Omega \\
    \Delta v+\lambda(1+u)^p=0, &  \operatorname{in} \quad\Omega\\
    u=v=0, & \operatorname{in} \quad\partial\Omega.
  \end{array}
\right.$$

If we denote $f_1(u,v)=v,$ $f_2(u,v)=\lambda(u+1)^p,$ we see that
$\frac{\partial f_1}{\partial v}=1>0$ and $\frac{\partial
f_2}{\partial u}=\lambda p(u+1)^{p-1}>0$. Therefore, the convexity
of $\Omega$, Lemma 5.1 of \cite{t}, and the moving plane method near
$\partial\Omega$ (as in the appendix of \cite{glw}) imply that there
exist $t_0>0$ and $\alpha$ which depends only on the domain
$\Omega$, such that $u(x-t\nu)$ and $v(x-t\nu)$ are nondecreasing
for $t\in[0,t_0],\quad\nu\in R^N$ satisfying $|\nu|=1$ and
$(\nu,n(x))\geq \alpha$ and $x\in \partial\Omega.$ Therefore, we can
find $\rho,\epsilon>0$ such that for any $x\in
\Omega_{\epsilon}:=\{z\in \Omega:~d(z,\partial\Omega)<\epsilon\}$
there exists a fixed-sized cone $\Gamma_x$ (with x as its vertex)
with

\begin{itemize}
  \item meas$(\Gamma_x)\geq \rho$,
  \item $\Gamma_x\subset \{z\in \Omega:~d(z,\partial\Omega)<2\epsilon\}$, and
  \item $u(y)\geq u(x)$ for any $y\in \Gamma_x.$
\end{itemize}

Then, for any $x\in \Omega_{\epsilon}$, we have

$$(1+u(x))^p\leq\frac{1}{\operatorname{meas}(\Gamma_x)}\int_{\Gamma_x}(1+u)^p
\leq\frac{1}{\rho}\int_{\Omega}(1+u)^p\leq C.$$

This  implies  that $(1+u(x))^p\leq C$, therefore $u(x)\leq C$.
\end{proof}

\noindent {\bf Remark:} By classical elliptic regularity theory,
$u(x)$ and its derivative up to fourth order are bounded on the
boundary by a constant independent of $u$. See \cite{w} for more
details. \vspace{0.75cm}

\subsection{Proof of Theorem \ref{theorem1.2}}
In the following, we will use the idea  in Section 2 to prove Theorem \ref{theorem1.2}.

First of all, multiplying  (\ref{1.4}) by $(u+1)^q$  and integration by parts, we have
\begin{equation}
\label{3.4}
\int_{\Omega}\lambda(u+1)^{p+q}=\int_{\Omega}\Delta^2u(u+1)^{q}=
\int_{\partial\Omega}\frac{\partial(\Delta u)}{\partial n}+\int_{\Omega}\Delta(u+1)\Delta(u+1)^{q}.
\end{equation}
Setting $v=u+1$, by direct calculations, we get
\begin{align}
\label{3.5}
\int_{\Omega}(\Delta v^{\gamma})^2=&\int_{\Omega}\gamma^2v^{2\gamma-2}(\Delta v)^2+\int_{\Omega}\gamma^2(\gamma-1)^2v^{2\gamma-4}|\nabla v|^4\nonumber\\
&+2\int_{\Omega}\gamma^2(\gamma-1)v^{2\gamma-3}\Delta v|\nabla v|^2,
\end{align}
\begin{equation}
\label{3.6}
\int_{\Omega}\Delta v\Delta v^{q}=\int_{\Omega}q(\Delta v)^2v^{q-1}+\int_{\Omega}q(q-1)|\nabla v|^2\Delta vv^{q-2}.
\end{equation}
From (\ref{3.4}), (\ref{3.5}) and (\ref{3.6}), we obtain
\begin{equation}
\label{3.7}
\int_{\Omega}(\frac{q}{\gamma^2}(\Delta v^{\gamma})^2-q(\gamma-1)^2|\nabla v|^4v^{2\gamma-4})+\int_{\partial\Omega}\frac{\partial(\Delta v)}{\partial n}=\int_{\Omega}\lambda v^{p+q}.
\end{equation}
For the second term in (\ref{3.7}), we have
\begin{align}
\label{3.8}
\int_{\Omega}|\nabla v|^4v^{2\gamma-4}&=\frac{1}{\gamma^4}\int_{\Omega}v^{-2\gamma}|\nabla v^{\gamma}|^4=\frac{1}{\gamma^4}\int_{\Omega}|\nabla v^{\gamma}|^2\nabla v^{\gamma}(-\nabla v^{-\gamma})\nonumber\\
&=\frac{1}{\gamma^4}\int_{\Omega}(-\nabla(\frac{|\nabla v^{\gamma}|^2\nabla v^{\gamma}}{v^{\gamma}})+\frac{\nabla(|\nabla v^{\gamma}|^2)\nabla v^{\gamma}}{v^{\gamma}}+\frac{|\nabla v^{\gamma}|^2\Delta v^{\gamma}}{v^{\gamma}})\nonumber\\
&=\frac{1}{\gamma^4}\int_{\Omega}\frac{\nabla(|\nabla v^{\gamma}|^2)\nabla v^{\gamma}+|\nabla v^{\gamma}|^2\Delta v^{\gamma}}{v^{\gamma}}-\frac{1}{\gamma}
\int_{\partial\Omega}v^{2\gamma-3}|\nabla v|^2\frac{\partial v}{\partial n}.\nonumber\\
\end{align}
Since the simple calculation implies that
\begin{equation}
\label{3.9} \frac{1}{\gamma^4}\int_{\Omega}\frac{|\nabla
v^{\gamma}|^2\Delta v^{\gamma}}{v^{\gamma}}
=\frac{\gamma-1}{\gamma}\int_{\Omega}v^{2\gamma-4}|\nabla
v|^4+\frac{1}{\gamma} \int_{\Omega}v^{2\gamma-3}|\nabla v|^2\Delta
v,
\end{equation}
by substituting (\ref{3.9}) into (\ref{3.8}), we get
\begin{equation}
\label{3.10}
\int_{\Omega}|\nabla v|^4v^{2\gamma-4}=\int_{\Omega}v^{2\gamma-3}|\nabla v|^2\Delta v+\frac{1}{\gamma^3}\int_{\Omega}\frac{\nabla(|\nabla v^{\gamma}|^2)\nabla v^{\gamma}}{v^{\gamma}}-\int_{\partial\Omega}|\nabla v|^2\frac{\partial v}{\partial n}.
\end{equation}
We now estimate the second term appeared on the right hand side of
(\ref{3.10}). From the proof of Lemma \ref{lemma2.3}, together with
the identity $\frac{1}{2} \Delta |\nabla v^\gamma|^2 = |\nabla^2
v^\gamma|^2 + <\nabla \Delta v^\gamma, \nabla v^\gamma>$, the
following inequality holds
\begin{eqnarray}
\label{3.11} \frac{1}{\gamma^3}\int_{\Omega}\frac{\nabla(|\nabla
v^{\gamma}|^2)\nabla v^{\gamma}}{v^{\gamma}} & \leq &
\frac{1}{2}\int_{\Omega}|\nabla
v|^4v^{2\gamma-4}+\frac{2}{\gamma^{2}}\int_{\Omega}(\Delta
v^{\gamma})^2\nonumber\\
& & + \frac{1}{\gamma^2} \int_{\partial\Omega} \frac{\partial
|\nabla v^\gamma|^2}{\partial n} - \frac{2}{\gamma^2} \int_{\partial
\Omega} (\Delta v^\gamma) \frac{\partial v^\gamma}{\partial n}.
\end{eqnarray}
By (\ref{3.10}) and (\ref{3.11}), thanks to the convexity of the
domain $\Omega$, we get
\begin{equation}
\label{3.12} \frac{1}{2}\int_{\Omega}|\nabla
v|^4v^{2\gamma-4}\leq\int_{\Omega}v^{2\gamma-3}|\nabla v|^2\Delta v
+\frac{2}{\gamma^{2}}\int_{\Omega}(\Delta v^{\gamma})^2-(2\gamma -
1)\int_{\partial\Omega}|\nabla v|^2\frac{\partial v}{\partial n}.
\end{equation}
For the first term on the right hand side of (\ref{3.12}), since
$v=u+1$, we have $\Delta v=\Delta u<0$ by maximal principle, and the
inequality $\Delta
v<-\sqrt{\frac{2\lambda}{p+1}}v^{\frac{p+1}{2}}<0$ by Lemma 3.2 of
\cite{cen}.  Thus
\begin{equation*}
\int_{\Omega}v^{2\gamma-3}|\nabla v|^2\Delta v\leq\int_{\Omega}-\sqrt{\frac{2\lambda}{p+1}}v^{2\gamma-3+\frac{p+1}{2}}|\nabla v|^2.
\end{equation*}
Moreover, we have
\begin{align*}
\int_{\Omega}-\sqrt{\frac{2\lambda}{p+1}}v^{2\gamma-3+\frac{p+1}{2}}|\nabla v|^2=&
-\int_{\Omega}\frac{\sqrt{\frac{2\lambda}{p+1}}}{2\gamma-2+\frac{p+1}{2}}\nabla(v^{2\gamma-2+\frac{p+1}{2}}\nabla v)
\\&+\int_{\Omega}\frac{\sqrt{\frac{2\lambda}{p+1}}}{2\gamma-2+\frac{p+1}{2}}v^{2\gamma-2+\frac{p+1}{2}}\Delta v.\end{align*}
For the second term on the right hand side of the above equality, using the inequality $\Delta v<-\sqrt{\frac{2\lambda}{p+1}}v^{\frac{p+1}{2}}<0$ again, we have
\begin{equation*}
\int_{\Omega}\frac{\sqrt{\frac{2\lambda}{p+1}}}{2\gamma-2+\frac{p+1}{2}}v^{2\gamma-2+\frac{p+1}{2}}\Delta v \leq
-\int_{\Omega}\frac{\frac{2\lambda}{p+1}}{2\gamma-2+\frac{p+1}{2}}v^{2\gamma+p-1}.
\end{equation*}
Hence, we obtain
\begin{align}
\label{3.13}
\int_{\Omega}v^{2\gamma-3}|\nabla v|^2\Delta v\leq
-\int_{\partial\Omega}\frac{\sqrt{\frac{2\lambda}{p+1}}}{2\gamma-2+\frac{p+1}{2}}\frac{\partial v}{\partial n}-\int_{\Omega}\frac{\frac{2\lambda}{p+1}}{2\gamma-2+\frac{p+1}{2}}v^{2\gamma+p-1},
\end{align}
where we used $v|_{\partial\Omega}=u+1|_{\partial\Omega}=1$, for the boundary term appeared in (\ref{3.4}), (\ref{3.12}) and (\ref{3.13}).  By the remark after Theorem \ref{theorem3.1}, we find that there exists a constant $C$ (the constant $C$ appeared now and later in this section is independent of $u$), such that
\begin{equation}
\label{3.14}
\int_{\partial\Omega}(|\nabla u|^2|\frac{\partial u}{\partial n}|+|\frac{\partial(\Delta u)}{\partial n}|+|\frac{\partial u}{\partial n}|)\leq C.
\end{equation}
Combining (\ref{3.7}), (\ref{3.12}), (\ref{3.13}) and (\ref{3.14}), we get
\begin{equation}
\label{3.15}
(1-4(\gamma-1)^2)\int_{\Omega}(\Delta (u+1)^{\gamma})^2+(\frac{8\lambda\gamma^2(\gamma-1)^2}{(4\gamma+p-3)(p+1)}-\frac{\lambda\gamma^2}{q})\int_{\Omega}(u+1)^{p+q}\leq C.
\end{equation}
If $(1-4(\gamma-1)^2)>0$, $p(1-4(\gamma-1)^2)+\frac{8\gamma^2(\gamma-1)^2}{(4\gamma+p-3)(p+1)}-\frac{\gamma^2}{q}>0$ and $u$ is a stable solution to the equation (\ref{1.4}), we have
\begin{equation*}
(p(1-4(\gamma-1)^2)+\frac{8\gamma^2(\gamma-1)^2}{(4\gamma+p-3)(p+1)}-\frac{\gamma^2}{2\gamma-1})\int_{\Omega}(u+1)^{p+q} \leq \frac{C}{\lambda}.
\end{equation*}
This leads to $u+1\in L^{p+q}$.

 If $p+q > \frac{(p-1)n}{4}$, then classical regularity theory implies that $u\in L^{\infty}(\Omega)$.

Therefore we have established the bound of extremal solutions of
(\ref{1.4})  if
$$ p(1-4(\gamma-1)^2)+\frac{8\gamma^2(\gamma-1)^2}{(4\gamma+p-3)(p+1)}-\frac{\gamma^2}{q}>0 $$
and
$$p<\frac{8\gamma+n-4}{n-4}.$$

By Lemma \ref{lemma5.1} and Theorem 3.8 of \cite{wy}, we prove the
extremal solution $u^*$, the unique solution of equation (\ref{1.4})
(where $\lambda=\lambda^*$) is bounded provided that
\begin{enumerate}
  \item $n\leq 8$, $p>1$,
  \item $9\leq n\leq 19$, there exists $\varepsilon_n>0$ such that for any $1<p<\frac{n}{n-8}+\varepsilon_n,$
  \item $n \geq 20$,  $1<p<1+\frac{8p^*}{n-4}$.\quad (${p^*}$ is defined as before )
\end{enumerate}
\vspace{1.5cm}

\setcounter{equation}{0}
\section{appendix}
In this appendix, we study the following inequalities
\begin{equation}
\label{5.1}
p(1-4(\gamma-1)^2)-\frac{\gamma^2}{2\gamma-1}+\frac{8\gamma^2(\gamma-1)^2}{(4\gamma-3+p)(p+1)}>0,
\end{equation}
\begin{equation}
\label{5.2}
p<\frac{8\gamma+n-4}{n-4}.
\end{equation}
In order to get a better range of the power $p$ from (\ref{5.1}) and (\ref{5.2}), it is necessary for us to study the following equation (Letting $p=\frac{8\gamma+n-4}{n-4}$ in (\ref{5.1})):
\begin{equation}
\label{5.3}
\frac{8\gamma+n-4}{n-4}(1-4(\gamma-1)^2)-\frac{\gamma^2}{2\gamma-1}+
\frac{8\gamma^2(\gamma-1)^2}{(4\gamma-3+\frac{8\gamma+n-4}{n-4})(\frac{8\gamma+n-4}{n-4}+1)}=0.
\end{equation}
We can only consider the behavior of (\ref{5.3}) for $\gamma\in(1,\frac{3}{2})$.
Through tedious computations, we see the following equation which appeared in the introduction is the simplified form of (\ref{5.3}). As a consequence, they have same roots in $(1,\frac{3}{2}):$
\begin{align}
\label{5.4}
& 512(2-n)\gamma^6+4(n^3-60n^2 +670n-1344)\gamma^5-2(13n^3-424n^2+3064n-5408)\gamma^4\nonumber\\
& +2(27n^3-572n^2 + 3264n-5440)\gamma^3- (49n^3-772n^2 +3776n-5888)\gamma^2\nonumber\\
& +4(5n^3-66n^2+288n-416)\gamma-3(n^3 -12n^2 + 48 n - 64) = 0.
\end{align}
We denote the left hand side of the equation (\ref{5.3}) by
$h(\gamma)$. Notice that if $\gamma = \frac{n-4}{n-8}$, then $p =
\frac{n}{n-8}$ and $\gamma - 1 = \frac{4}{n-8}$. Hence
$$h(\frac{n-4}{n-8}) = \frac{8}{n-8}[ n^4 - 18 n^3 - 56 n^2 + 384 n - 512].$$  In fact, if $n = 20$, then $h(\frac{4}{3}) = 512 > 0$.
On the other hand, it is also easy to see that $h(\frac{3}{2})<0,$
while it is obvious that
$(4\gamma-3+\frac{8\gamma+n-4}{n-4})(\frac{8\gamma+n-4}{n-4}+1)>0$
and $(2\gamma-1)>0$ when $\gamma \in (\frac{n-4}{n-8},\frac{3}{2})$.
Therefore, by continuity, equation (\ref{5.3}) possesses a root in
$(\frac{n-4}{n-8},\frac{3}{2}).$ We denote the smallest root of
(\ref{5.3}) which is greater than $\frac{n-4}{n-8}$ by $p^*$. Once
we pick out a $\gamma$ from the interval $(\frac{n-4}{n-8},p^*)$,
$h(\gamma)$ is of course positive. By continuity, we can find a
small positive number $\delta$ such that, the inequality
$p(1-4(\gamma-1)^2)-\frac{\gamma^2}{2\gamma-1}+\frac{8\gamma^2(\gamma-1)^2}{(4\gamma-3+p)(p+1)}>0$
holds when $p
\in(\frac{8\gamma+n-4}{n-4}-\delta,\frac{8\gamma+n-4}{n-4})$. So, we
conclude that when $\gamma$ runs in the whole interval
$(\frac{n-4}{n-8},p^*)$, the power $p$ can be in the whole interval
$(\frac{n}{n-8},1+\frac{8p^*}{n-4}).$  We summarize the result as
follows:
\begin{lemma}
\label{lemma5.1}
When $n\geq20$, we have $p$ which satisfies (\ref{5.1}) and (\ref{5.2}) can range in $(\frac{n}{n-8},1+\frac{8p_n}{n-4})$ and this interval is not empty.
\end{lemma}
\vspace{2cm}

\noindent
{\bf Acknowledgments:} The first  author was supported from an Earmarked grant (``On Elliptic Equations with Negative Exponents'') from RGC of Hong Kong.

\bibliographystyle{plain}

\end{document}